\documentclass[11pt]{amsart}
\usepackage{amssymb,amsmath,amsfonts,amsthm,mathrsfs,fullpage}
\usepackage{color}

\numberwithin{equation}{section}

\renewcommand{\AA}{\mathbb A}

\newcommand{\FF}{\mathbb F}
\newcommand{\GG}{\mathbb G}

\newcommand{\QQ}{\mathbb Q}

\newcommand{\ZZ}{\mathbb Z} 

\newcommand{\Zhat}{\widehat\ZZ}

\newcommand{\OO}{\mathcal O}
\newcommand{\calP}{\mathcal P}

\newcommand{\m}{\mathfrak m}
\newcommand{\aA}{\mathfrak a}

\newcommand{\p}{\mathfrak p}

\def\Tr{\operatorname{Tr}}
 \def\Gal{\operatorname{Gal}}
\def\End{\operatorname{End}}

\def\ab{{\operatorname{ab}}}

\def\tr{\operatorname{tr}}
\def\Tr{\operatorname{Tr}}
 
\def\Gal{\operatorname{Gal}}

\def\sl{\mathfrak{sl}}
  
\def \GL{\operatorname{GL}}  
\def \PGL{\operatorname{PGL}}
\def \SL{\operatorname{SL}}

\def\Aut{\operatorname{Aut}} 
\def\End{\operatorname{End}}
\def\Frob{\operatorname{Frob}}

\def\tr{\operatorname{tr}}

 \def \Aut {\operatorname{Aut}}

\definecolor{purple}{rgb}{1,0,1}

\def\bbar#1{\setbox0=\hbox{$#1$}\dimen0=.2\ht0 \kern\dimen0 
\overline{\kern-\dimen0 #1}}
 
\newcommand{\Kbar}{\bbar{K}} 
 
\newcommand{\FFbar}{\overline{\FF}}

\def\sep{{\operatorname{sep}}}
\def\un{{\operatorname{un}}}

\newcommand{\defi}[1]{\textsf{#1}} 
\newtheorem{thm}{Theorem}[section]
\newtheorem{lemma}[thm]{Lemma}

\newtheorem{prop}[thm]{Proposition}

\theoremstyle{definition}
\newtheorem{defn}[thm]{Definition}

\theoremstyle{remark}
\newtheorem{remark}[thm]{Remark}

\newenvironment{romanenum}{\hfill \begin{enumerate} }{\end{enumerate}}
\newenvironment{alphenum}{\hfill \begin{enumerate} }{\end{enumerate}}

\definecolor{webcolor}{rgb}{0.8,0,0.2}
\definecolor{webbrown}{rgb}{.6,0,0}
\usepackage[
        colorlinks,
        linkcolor=webbrown,  filecolor=webcolor,  citecolor=webbrown, 
        backref,
        pdfauthor={David Zywina}, 
        pdftitle={Drinfeld modules with maximal Galois action on their torsion points},
]{hyperref}
\usepackage[alphabetic,backrefs,lite]{amsrefs} 

\begin{document}
\title{Drinfeld modules with maximal Galois action on their torsion points}
\author{David Zywina}
\address{Department of Mathematics and Statistics, Queen's University, Kingston, ON  K7L~3N6, Canada}
\email{zywina@mast.queensu.ca}
\urladdr{http://www.mast.queensu.ca/\~{}zywina}

\date{\today}

\begin{abstract}
To each Drinfeld module over a finitely generated field with generic characteristic, one can associate a Galois representation arising from the Galois action on its torsion points.   Recent work of Pink and R\"utsche has described the image of this representation up to commensurability.   Their theorem is qualitative, and the objective of this paper is to complement this theory with a worked out example.   In particular, we give examples of Drinfeld modules of rank 2 for which the Galois action on their torsion points is as large as possible.   We will follow the approach that Serre used to give explicit examples of his openness theorem for elliptic curves.  Using our specific examples, we will numerically test analogues of some well-known elliptic curve conjectures.
\end{abstract}

\maketitle

\section{Introduction}

Let $\FF_q$ be a finite field with $q$ elements.    Let $A$ be the ring $\FF_q[T]$ and let $F$ be its fraction field.  For a given field extension $K$ of $\FF_q$, let $\Kbar$ be an algebraic closure of $K$ and let $K^\sep$ be the separable closure of $K$ in $\Kbar$.  Denote by $G_K=\Gal(K^\sep/K)$ the absolute Galois group of $K$.

\subsection{Drinfeld modules and Galois representations}
We now give enough background in order to state and explain our theorem.  For an in-depth introduction to Drinfeld modules, see~\cite{MR1423131,MR902591,MR0384707}.

Let $K\{\tau\}$ be the ring of skew polynomials; i.e., the ring of polynomials in the indeterminate $\tau$ with coefficients in $K$ that satisfy the commutation rule $\tau\cdot c = c^q \tau$ for $c\in K$.   One can identify $K\{\tau\}$ with a subring of $\End(\GG_{a,\, K})$ by identifying $\tau$ with the Frobenius map $X\mapsto X^q$.    
A \defi{Drinfeld $A$-module} over $K$ is a homomorphism 
\[
\phi\colon A \to K\{\tau\}, \quad a \mapsto \phi_a
\]
of $\FF_q$-algebras whose image is not contained in $K$.     The Drinfeld module $\phi$ is determined by 
$\phi_T = \sum_{i=0}^r a_i \tau^i$ where $a_i\in K$ and $a_r\neq 0$; the positive integer $r$ is called the \defi{rank} of $\phi$.  

Let $\partial_0 \colon K\{\tau\} \to K$ be the ring homomorphism $\sum_{i} a_i \tau^i \mapsto a_0$.    The \defi{characteristic} of $\phi$ is the kernel $\p_0$ of the homomorphism $\partial_0 \circ \phi \colon A \to K$.  If $\p_0$ is the zero ideal, then we say that $\phi$ has \defi{generic characteristic} and we may then view $K$ as an extension of $F$.

The Drinfeld module $\phi$ endows $K^\sep$ with an $A$-module structure, i.e., $a\cdot x = \phi_a(x)$ for $a\in A$ and $x\in K^\sep$.   We shall write ${}^\phi \!K^\sep$ if we wish to emphasize this action.  For a non-zero ideal $\aA$ of $A$, the \defi{$\aA$-torsion} of $\phi$ is
\[
\phi[\aA] :=\{x\in {}^\phi\!K^\sep : a \cdot x = 0 \text{ for all } a\in \aA \} =  \{x \in K^\sep :  \phi_a(x) = 0 \text{ for all } a\in \aA \}.
\]
If $\aA$ is relatively prime to the characteristic $\p_0$, then $\phi[\aA]$ is a free $A/\aA$-module of rank $r$.   

For the rest of the section, assume that $\phi$ has generic characteristic.  The absolute Galois group $G_K$ acts on $\phi[\aA]$ and respects the $A$-module structure.  This action can be re-expressed in terms of a Galois representation
\[
\bbar\rho_{\phi, \aA} \colon G_K \to \Aut(\phi[\aA]) \cong \GL_r(A/\aA).
\]
Let $\p$ be a place of $A$.  If $\phi$ has good reduction at $\p$ and $\p \nmid \aA$, then the representation $\bbar\rho_{\phi,\aA}$ is unramified at $\p$ (one can use this as a definition of good reduction, it will agree with the later definition).

For each non-zero prime ideal $\lambda$ of $A$, we have a Galois representation
\[
\rho_{\phi,\lambda} \colon G_K \to \Aut\Big(\varinjlim_i \phi[\lambda^i]\Big) \cong \GL_r(A_\lambda)
\]
where $A_\lambda$ is the $\lambda$-adic completion of $A$.  These representations have properties similar to the familiar $\ell$-adic representations attached to elliptic curves.   For example,  if $\phi$ has good reduction at $\p \nmid \lambda$, then $P_{\phi,\p}(x):=\det(xI - \rho_{\phi,\lambda}(\Frob_\p))$ is a polynomial with coefficients in $A$ that does not depend on $\lambda$.  In particular, $a_{\p}(\phi):=\tr(\rho_{\phi,\lambda}(\Frob_\p))$ is an element of $A$ that does not depend on $\lambda$.  Combining all the representations together, we obtain a single Galois representation 
\[
\rho_\phi\colon G_K \to \Aut\Big(\varinjlim_\aA \phi[\aA]\Big) \cong \GL_r(\widehat{A})
\]
where $\widehat{A}$ is the profinite completion of $A$.    

\subsection{Open image theorem}
Pink and R\"utsche have described the image of $\rho_\phi$ up to  commensurability \cite{MR2499412}.  For simplicity, we only state the version for which $\phi$ has no extra endomorphisms.  Recall that the ring $\End_{\Kbar}(\phi)$ of \defi{endomorphisms} is the centralizer of $\phi(A)$ in $\Kbar\{\tau\}$.

\begin{thm}[Pink-R\"utsche] \label{T:MT}
Let $\phi$ be a Drinfeld $A$-module of rank $r$ over a finitely generated field $K$.   Assume that $\phi$ has generic characteristic and that $\End_{\Kbar}(\phi)=\phi(A)$.   Then the image of
\[
\rho_\phi \colon G_K \to \GL_r(\widehat{A})
\]
is open in $\GL_r(\widehat{A})$.  Equivalently, $\rho_\phi(G_K)$ has finite index in $\GL_r(\widehat{A})$.   
\end{thm}

\subsection{An explicit example}
Theorem~\ref{T:MT} is qualitative in nature since it only describes the group $\rho_\phi(G_K)$ up to commensurability (it is unclear from the proof if it is feasible to actually compute the group $\rho_\phi(G_K)$ in general).    Except for the rank one case, which resembles the classical theory of complex multiplication, the author is unaware of any worked out examples in the literature.\\

The main objective of this paper is to compute the image of $\rho_\phi$ for an explicit example.  This example also proves the existence of Drinfeld modules of rank $2$ for which the Galois action on its torsion points is maximal.

\begin{thm}  \label{T:MT example}
Let $q\geq 5$ be an odd prime power.   Let $\varphi \colon \FF_q[T] \to \FF_q(T)\{\tau\}$ be the Drinfeld module of rank $2$ for which
\[
\varphi_T = T + \tau -T^{q-1} \tau^2.
\]   
Then the Galois representation 
\[
\rho_\varphi \colon G_{\FF_q(T)} \to \GL_2\!\big(\widehat{\FF_q[T]}\big)
\]
is surjective.  Moreover, $\rho_\varphi\big(G_{\FFbar_q(T)}\big)=\GL_2\!\big(\widehat{\FF_q[T]}\big)$.
\end{thm}

We will use our example to investigate various conjectures in \S\ref{S:conjectures}.

\subsection{Elliptic curves}
We now discuss the analogous theory of elliptic curves which strongly influences the proof of Theorem~\ref{T:MT} and the methods of this paper.  Let $E$ be an elliptic curve defined over a number field $K$.   For each positive integer $m$, the Galois action on the $m$-torsion points $E[m]$ of $E(K^\sep)$ gives a representation
\[
\bbar\rho_{E,m} \colon G_K \to \Aut(E[m])\cong \GL_2(\ZZ/m\ZZ).
\]
Let $\p$ be a non-zero prime ideal for which $E$ has good reduction, then there is a unique polynomial $P_{E,\p}(x)=x^2 - a_\p(E) x +N(\p) \in \ZZ[x]$ such that $P_{E,\p}(x) \equiv \det(I - \bbar\rho_{E,m}(\Frob_\p)) \pmod{m}$ when $\p\nmid m$.   Combining the representations $\bbar\rho_{E,m}$ together, we obtain a single Galois representation
\[
\rho_E \colon G_K \to \Aut\Big(\varinjlim_m E[m]\Big)\cong \GL_2(\Zhat).
\]
In 1972, Serre \cite{MR0387283} showed that if $\End_{\Kbar}(E)\cong \ZZ$, then $\rho_E$ is open in $\GL_2(\Zhat)$; this is a clear analogue of Theorem~\ref{T:MT}.  Earlier, Serre had shown that $\rho_{E}$ has open image if $E$ has non-integral $j$-invariant (cf.~\S3.2 of \cite{MR0263823}*{Chapter~IV}).  What makes the non-integral $j$-invariant case easier is that, using the theory of \emph{Tate curves}, one can show that the image of $\bbar\rho_{E,\ell}\colon G_K \to \GL_2(\ZZ/\ell\ZZ)$ contains an element of order $\ell$ for all but finitely many primes $\ell$ (cf.~the proposition of \cite{MR0263823}*{Chapter~IV Appendix A.1.5}).    Serre has given worked out examples of $\rho_E(G_K)$ for several non-CM elliptic curves over $K=\QQ$ (cf.~  \cite{MR0387283}*{\S5.5}).   The first example with surjective $\rho_E$ was given by A.~Greicius \cite{MR2778661}.

\subsection{Overview}

In \S\ref{S:conjectures}, which is independent of the rest of the paper, we mention three well-known conjectures for elliptic curves.   The heuristics and resulting predictions for these conjectures depend on the Galois action on the curve's torsion points.  Having worked out the Galois image for a specific Drinfeld module of rank $2$, we can now start doing numerical experiments on the analogous Drinfeld module statements.

The remainder of the paper is dedicated to the proof of Theorem~\ref{T:MT example}.

In \S\ref{S:the determinant}, we prove that the character $\det\circ \rho_\varphi \colon G_F \to \widehat{A}^\times$ is surjective.  This will be accomplished by first recognizing that $\det\circ \rho_\varphi$ is	 the representation $\rho_C$ associated with the Carlitz module $C$; this particular Drinfeld module has been extensively studied.

In \S\ref{S:The Drinfeld-Tate uniformization}, we shall recall the \emph{Tate uniformization} of a Drinfeld module (this is the analogue of the usual Tate uniformization of elliptic curves over non-archimedean local fields).    We can then apply this theory to our Drinfeld module $\varphi$ at the place $(T)$ where it has bad and stable reduction.  The main application of the section is that for every non-zero prime ideal $\lambda$ of $A$, $\bbar\rho_{\varphi,\lambda}(G_F)$ contains a $p$-Sylow subgroup of $\GL_2(A/\lambda)$ where $p$ is the prime dividing $q$.

In \S\ref{S:irreducibility}, we prove that $\varphi[\lambda]$ is an irreducible $\FF_{\lambda}[G_F]$-module for all $\lambda$.  If $\varphi[\lambda]$ is reducible, then we can understand the semi-simplification of the action of $G_F$ on $\varphi[\lambda]$ in terms of two characters $\chi,\chi' \colon G_F \to \FF_\lambda^\times$.  Using our knowledge of $\bbar\rho_{\varphi,\lambda}$, we will describe the possibilities for the pair $\{\chi,\chi'\}$, and then derive a contradiction based upon traces of Frobenii.

Having shown that the image of $\rho_E(G_F)$ is ``large'' in these three different contexts, we will then combine them to prove surjectivity. That the residual representations $\bbar\rho_{\varphi,\lambda}$ are surjective, will follow quickly.   In contrast to the elliptic curve situation, it takes some serious work to prove that these representations are independent.   For elliptic curves, one makes use of the easy fact that the groups $\SL_2(\ZZ_\ell)$ have no common quotients; this fails for the groups $\SL_2(A_\lambda)$ (just consider $\lambda$ with the same degree).  The required group theory is contained in a short appendix.

\subsection*{Notation}

We fix throughout an odd prime power $q\geq 5$.  We let $A$ be the ring $\FF_q[T]$ and we let $F$ be the fraction field $\FF_q(T)$.  For a positive integer $n$, let $\pi_A(n)$ be the number of monic irreducible polynomials of degree $n$ in $A$.

We will usually denote a non-zero prime ideal of $A$ by $\lambda$, which we will also call a \defi{finite place of $F$.}   Since $A=\FF_q[T]$ is a PID, we will occasionally identify $\lambda$ with its monic irreducible generator.    We shall denote the residue field $A/\lambda$ by $\FF_\lambda$.  Let $A_\lambda$ and $F_\lambda$ be the $\lambda$-adic completion of $A$ and $F$, respectively.  For each non-zero ideal $\aA\subseteq A$, denote the cardinality of $A/\aA$ by $N(\aA)$.  

\section{Some conjectures} \label{S:conjectures}

Throughout this section, let $E$ be a non-CM elliptic curve over $\QQ.$   For each prime $p$ of good reduction, we have a finite abelian group $E(\FF_p)$.  Several number theoretic conjectures deal with the asymptotics of primes $p$ for which $E(\FF_p)$ has a fixed property.   In \S\ref{SS:cyclic}, we discuss a conjecture that predicts the distribution of $p$ for which $E(\FF_p)$ is cyclic.    In \S\ref{SS:Koblitz Drinfeld}, we discuss a conjecture of Koblitz that predicts the distribution of $p$ for which $E(\FF_p)$ has prime cardinality.   (These two conjectures are interesting in part because of public key cryptography where it is useful to have a point in $E(\FF_p)$ that generates a group of large prime order.)   Finally in \S\ref{SS:Lang Trotter}, we mention a conjecture of Lang and Trotter on the number of primes $p$ for which $a_p(E)$ takes a fixed value $t$.   The precise Drinfeld module analogue for the Lang-Trotter conjecture is the least clear.  We will not venture a full conjecture here, but hope to return to it in the future. \\

Before continuing, we recall a little more about the arithmetic of Drinfeld modules.   For simplicity, we restrict our attention to a Drinfeld module
\[
\phi\colon A\to F\{\tau\}
\]
of rank 2 such that $\partial_0\circ \phi \colon A\to F$ is the inclusion map and $\End_{\bbar{F}}(\phi)=\phi(A)$;  our numerical data will all be for the Drinfeld module $\varphi$ of Theorem~\ref{T:MT example} with $q=5$.    Let $\p$ be a non-zero prime ideal of $A$ for which $\phi$ has good reduction; we will sometimes identify $\p$ with its monic irreducible generator.  Reduction mod $\p$ induces a Drinfeld module $\phi_\p \colon A \to \FF_\p\{\tau\}$ where $\FF_\p=A/\p A$.   We shall denote by ${}^\phi\FF_\p$, the group $\FF_\p$ equipped with the $A$-module action coming from $\phi_\p$.   Note that as an $A$-module, ${}^\phi\FF_\p$ need not be isomorphic to $A/\p A$.  For conjectures concerning the structure of the groups $E(\FF_p)$, we will instead consider the structure of the $A$-modules  ${}^\phi\FF_\p$.   As an $A$-module ${}^\phi\FF_\p$ is isomorphic to 
\[
A/d_\p A \times A/d_\p e_\p A 
\]
where $d_\p$ and $e_\p$ are unique monic polynomials in $A$.  The \defi{Euler-Poincar\'e characteristic} of ${}^\phi\FF_\p$ is the ideal
\[
\chi_\phi(\p) = d_\p^2 e_\p A.
\]
We can relate the Euler characteristic to our Galois representations by noting that $\chi_\phi(\p)$ is the ideal of $A$ generated by $P_{\phi,\p}(1)$.  For conjectures about the orders $|E(\FF_p)|=P_{E,p}(1)$, the analogous object of study is $\chi_\phi(\p)$.

Let $D_\phi$ be the index of $\rho_\phi(G_{F\FFbar_q})$ in $\rho_\phi(G_F)$; it is known to be finite.   By Theorem~\ref{T:MT example}, we have $D_\varphi=1$.  Restriction gives an exact sequence
\[
1 \to G_{F\FFbar_q} \hookrightarrow G_F \to \Gal(\FFbar_{q}/\FF_q)\xrightarrow{\overset{\deg}{\sim}} \Zhat \to 1
\]
where $\deg(\Frob_\p)$ equals the degree of $\p$ for irreducibles $\p$ of $A$.  For an integer $n$ and a non-zero ideal $\aA$ of $A$, we define $\bbar\rho_{\phi,\aA}(G_F)_n$ to be the image under $\bbar\rho_{\phi,\aA}$ of $\{\sigma \in G_F : \deg(\sigma)\equiv n \pmod{D_\phi}\}$.   

\subsection{Cyclicity of reductions modulo $\p$} \label{SS:cyclic}
For primes $p$ of good reduction, the group $E(\FF_p)$ is isomorphic to $\ZZ/d_p\ZZ \times \ZZ/d_pe_p\ZZ$ for unique positive integers $d_p$ and $e_p$.   It is natural to ask how often $E(\FF_p)$ is cyclic ($d_p=1$), that is, describe the asymptotics of the function
\[
f_E(x) := \#\{p\leq x : E(\FF_p) \text{ is cyclic} \}.
\]
Serre \cite{Serre-resume1977-1978} showed that, assuming the Generalized Riemann Hypothesis (GRH), one has
\begin{equation} \label{E:cyclic equation}
f_E(x) \sim c_E \frac{x}{\log x}
\end{equation}
as $x\to + \infty$, where $c_E := \sum_{m\geq 1} \mu(m)/[\QQ(E[m]):\QQ]$ (see \S5 of \cite{MR698163} for a proof and \S6 for an unconditional result in the CM case).  If $c_E=0$, then we interpret this as meaning that $f_E(x)$ is uniformly bounded.  It is still unknown if (\ref{E:cyclic equation}) holds unconditionally.  If $c_E>0$, then Murty and Gupta \cite{MR1055716} showed that $f_E(x) \gg_E x/\log^2 x$, and the constant $c_E$ is positive if and only if $\QQ(E[2])\neq\QQ$.  For more background and progress on the conjecture, see \cite{MR2099195}.\\

The Drinfeld module analogue has been formulated and proven by W.~Kuo and Y.-R.~Liu \cite{MR2559117}.    We shall say that ${}^\phi\FF_\p$ is \defi{cyclic} if it is isomorphic as an $A$-module to $A/wA$ for some non-constant element $w$ of $A$; equivalently, $d_\p=1$.   For each positive integer $n$, we let $f_\phi(n)$ be the number of monic irreducible polynomials $\p$ of degree $n$ in $A$ for which ${}^\phi\FF_\p$ is cyclic.  Kuo and Liu have shown that
\[
\lim_{n\to + \infty} \frac{f_\phi(n)}{\pi_A(n)} = c_\phi(n)
\]
where $c_\phi(n)$ is an explicit constant whose value depends only on $n$ modulo $D_\phi$ (recall that in the function field setting the analogue of GRH is known to hold).  They also give a good bound on the error term, and prove a similar result for rank $\geq 3$.   Theorem~\ref{T:MT} plays a vital role in their proof.\\

Let us now consider our Drinfeld module $\varphi\colon A\to F\{\tau\}$.  Since $D_\varphi=1$, we can simply write $c_\varphi$ for the constant.  The constant given in \cite{MR2559117} is
\[
c_\varphi= \sum_{m \in A \text{ monic}} \frac{\mu_q(m)}{[F(\varphi[m]):F]}
\]
where $\mu_q$ is the M\"obius function for $A$ (i.e., the multiplicative function that vanishes on polynomials with a multiple irreducible factor and is $- 1$ for irreducible polynomials).  By Theorem~\ref{T:MT example}, we find that ${\mu_q(m)}/{[F(\varphi[m]):F]}={\mu_q(m)}/{|\!\GL_2(A/mA)|}$ is a multiplicative function, hence
\[
c_\varphi= \prod_{\substack{\lambda \in A\\\text{monic irreducible}}} \Big(1 - \frac{1}{|\!\GL_2(A/\lambda A)|} \Big) 
= \prod_{d=1}^\infty \Big(1 - \frac{1}{q^d(q^d-1)^2(q^d+1)} \Big)^{\pi_A(d)}.
\]
This expression for $c_\varphi$ is easy to estimate (note that $\pi_A(d)$ equals $\frac{1}{d} \sum_{e|d} \mu(e) q^{d/e}$).

For $q=5$, we find that $c_\varphi=0.989600049329883\ldots$.  In Table 1, we see that, even for small $n$, the ratio $f_\varphi(n)/\pi_A(n)$ is well approximated by $c_\varphi$.
\begin{table}[htdp] 
\begin{center}\begin{tabular}{c|c|c|c} $n$ & $f_\varphi(n)$ & $\pi_A(n)$ & $f_\varphi(n)/\pi_A(n)$ \\\hline\hline  
2 & 10 & 10 & 1 \\\hline
3 & 40 & 40 & 1 \\\hline
4 & 150 & 150 & 1 \\\hline
5 & 618 & 624 & $0.99038461\ldots$ 
\\\hline
 6 & 2554 & 2580 & $0.98992248\ldots$ 
 \\\hline 7 & 11069 & 11160 & $0.99184587\ldots$ 
 \\\hline 8 & 48270 & 48750 & $0.99015384\ldots$ 
 \\\hline 9 & 214807 & 217000 & $0.98989400\ldots$ 
 \\\hline 10 & 966135 & 976248 & $0.98964095\ldots$ 
 \\\hline 11 & 4392845 & 4438920 & $0.98962022\ldots$ 
 \\ \end{tabular} \caption{Values of $f_\varphi(n)$ for small $n$ where $\varphi\colon \FF_5[T] \to \FF_5(T)\{\tau\}$ is the Drinfeld module for which $T\mapsto  T + \tau - T^{4}\tau^2$.}
\end{center}
\label{Table:cyclic}
\end{table}


\subsection{Koblitz conjecture} \label{SS:Koblitz Drinfeld}
Let $\calP_E(x)$ be the number of primes $p\leq x$ of good reduction for which $|E(\FF_p)|$ is prime. Koblitz conjectured that there is a constant $C_E$ such that $\calP_E(x) \sim C_E \dfrac{x}{(\log x)^2}$ as $x\to +\infty,$ where $C_E\geq 0$ is an explicit constant.   The constant, as revised in \cite{Zywina-Koblitz}, is
\[
C_E = \lim_{m\to + \infty} \dfrac{|\{ g \in \bbar\rho_{E,m}(G_\QQ) : \det(I-g) \in (\ZZ/m\ZZ)^\times \}|/|\bbar\rho_{E,m}(G_\QQ)|}{\prod_{p|m}(1-1/p)}
\]
where $m$ runs over positive square-free integers ordered by divisibility.\\

Let $\calP_\phi(n)$ be the number of monic irreducible polynomials $\p$ of degree $n$ in $A$ for which $\chi_\phi(\p)$ is a prime ideal (equivalently, ${}^\phi\FF_\p$ is a simple $A$-module). We conjecture that
\begin{equation} \label{E:Koblitz Drinfeld}
\calP_\phi(n) \sim C_{\phi}(n) \frac{q^n}{n^2}
\end{equation}
as $n\to+ \infty$, where 
\[
C_\phi(n) = \lim_{m} \dfrac{|\{ g \in \bbar\rho_{\phi,m}(G_F)_n : \det(I-g) \in (A/mA)^\times \}|/|\bbar\rho_{\phi,m}(G_F)_n |}{\prod_{\lambda|m,\, \lambda \text{ prime}} (1- 1/N(\lambda))}
\]
and the $m$ run over monic polynomials of $A$ ordered by divisibility.  The constant $C_\phi(n)$ is well-defined (one uses Theorem~\ref{T:MT} to check convergence) and depends only on the value of $n$ modulo $D_\phi$.  Analogues of Koblitz's conjecture for Drinfeld modules were first investigated in the Master's thesis of L.~Jain \cite{Jain-thesis2008}. \\

Let us give a crude heuristic for this conjecture.   A ``random'' monic polynomial of degree $n$ in $A$ should be irreducible with probability $\pi_A(n)/q^n \approx 1/n$.      So for each irreducible polynomial $\p$ of degree $n$, we expect the degree $n$ monic generator of $\chi_\phi(\p)$ to be irreducible with probability $1/n$.   Thus a naive estimate for $\calP_\phi(n)$ is $\pi_A(n)/n\approx q^n/n^2$.     What this heuristic is ignoring is that the Euler characteristics $\chi_A(\p)$ are not random ideals when it comes to congruences.     Fix a monic polynomial $m$ in $A$.  Then the ratio of irreducible polynomials $\p$ of degree $n$ for which $m$ is relatively prime to $\chi_\phi(\p)$ is approximately $|\{ g \in \bbar\rho_{\phi,m}(G_F)_n : \det(I-g) \in (A/mA)^\times \}|/|\bbar\rho_{\phi,m}(G_F)_n|$, while the naive ratio is $|(A/mA)^\times|/|A/mA| = \prod_{\lambda|m}(1- 1/N(\lambda))$.    So by taking into account congruences modulo $m$, we expect
\[
\dfrac{|\{ g \in \bbar\rho_{\phi,m}(G_F)_n : \det(I-g) \in (A/mA)^\times \}|/|\bbar\rho_{\phi,m}(G_F)_n|}{\prod_{\lambda|m,\, \lambda \text{ prime}} (1- 1/N(\lambda))}\, \frac{q^n}{n^2}
\]
to be a better estimate for $\calP_{\phi}(n)$.   The conjecture (\ref{E:Koblitz Drinfeld}) arises by letting $m$ run over more and more divisible elements of $A$ (it actually suffices to consider only squarefree $m$).\\

Now consider our Drinfeld module $\varphi$.  Since $D_\varphi=1$, we can simply write $C_\varphi$ for the constant.  Using Theorem~\ref{T:MT example}, one can compute as in \cite{Zywina-Koblitz} and show that
\begin{align*}
C_\varphi &=  \prod_\lambda  \dfrac{|\{ g \in \GL_2(A/\lambda)  : \det(I-g) \in (A/\lambda A)^\times \}|/|\GL_2(A/\lambda A)|}{ 1- 1/N(\lambda)}\\
&= \prod_{d\geq 1} \Big( 1 - \frac{q^{2d}-q^d - 1}{(q^d-1)^3(q^d+1)} \Big)^{\pi_A(d)}.
\end{align*}
Taking $q=5$, we find that $C_\varphi = 0.76075227630\ldots$.  
In Table 2, we numerically test our analogue of Koblitz's conjecture for small $n$.

\begin{table}[htdp] 
\begin{center}\begin{tabular}{c|c|c} $n$ & $P_\varphi(n)$ & $P_\varphi(n)/\Big(\frac{5^n}{n^2}\Big)$ \\\hline\hline  
2 & 5  & 0.8 \\\hline
3 & 10 & 0.72  \\\hline
4 & 41  & 1.0496  \\\hline
5 & 106  & 0.848  \\\hline
 6 & 317  & 0.730368  \\\hline 
 7 & 1194 & 0.7488768  \\\hline 
 8 & 4540 & 0.7438336  \\\hline 
 9 & 18534 &  0.768642048 \\\hline 
 10 & 74724 &  0.76517376  \\\hline
 11 & 307931 & 0.76308013056 \\
\end{tabular} \caption{Numerical evidence for the Koblitz conjecture for the Drinfeld module $\varphi\colon \FF_5[T] \to \FF_5(T)\{\tau\}$, $T\mapsto  T + \tau - T^{4}\tau^2$.}
\end{center}
\label{Table:Koblitz}
\end{table}

\subsection{Lang-Trotter} \label{SS:Lang Trotter}
Fix a positive integer $t$.   In \cite{MR0568299}*{Part I}, Lang and Trotter  conjectured that
\[
|\{ p \leq x: a_p(E) = t \}| \sim C_{E,t} \frac{\sqrt{x}}{\log x}
\]
as $x\to \infty$, where $C_{E,t}\geq 0$ is an explicit constant.  Again if $C_{E,t}=0$, we interpret the asymptotic as meaning that there are only finitely many primes $p$ such that $a_p(E)=t$.   Their predicted constant is
\[
C_{E,t} = \frac{2}{\pi} \, \lim_{m\to \infty} \frac{|\{g \in \bbar\rho_{E,m}(G_\QQ): \tr(g)\equiv t \pmod{m}\}|/| \bbar\rho_{E,m}(G_\QQ)|}{1/m}
\]
where the limit is over positive integers $m$ ordered by divisibility.    Note that the value of $C_{E,t}$ depends only on the image of the representation $\rho_E\colon G_\QQ \to \GL_2(\Zhat)$.  In particular, their  conjecture implies that there are infinitely many primes $p$ for which $a_p(E)=t$ if and only if if there are \emph{no congruence obstructions} (i.e., for each $m$, there is one, and hence infinitely many, primes $p$ of good reduction such that $a_p(E)\equiv t \pmod{m}$).\\

Now consider the Drinfeld module analogue.  Fix an element $t\in A$.  Let $\calP_{\phi,t}(n)$ be the number of monic irreducible polynomials of degree $n$ in $A$ for which $a_\p(\phi)= t$.   We conjecture that there is a positive integer $M_\phi$ such that
\begin{equation} \label{E:LT Drinfeld}
\calP_{\phi,t}(n) \sim C_{\phi,t}(n)  \frac{q^{n/2}}{n}
\end{equation}
as $n\to +\infty$, where $C_{\phi,t}(n)$ are constants whose value depends only on $n$ modulo $M_\phi$.  As we will see, the conjecture is in general false with $M_\phi=D_\phi$.  This conjecture has been proven ``on average'' for $t=0$ by C.~David \cite{MR1373559}, and the main theorem there suggests that $M_\phi=2$ for a ``random'' $\phi$.  \\

Consider our Drinfeld module $\varphi \colon \FF_q[T] \to \FF_{q}(T)\{\tau\}$ from Theorem~\ref{T:MT example}.  The following proposition shows that when $q$ is congruent to $1$ modulo $4$ and $n$ is a sufficiently large \emph{even} integer, we will always have $C_{\varphi,t}(n)=0$.   

\begin{prop} \label{P:obstruction}
Assume that $q\equiv 1 \pmod{4}$.  If $n$ is an even integer and $\p\neq (T)$ is an irreducible polynomial in $A$ of degree $n$, then $a_\p(t)$ has degree $n/2$.
\end{prop}
\begin{proof}
Let $L$ be the field $\FF_q[T]/\p$.  Since $n$ is even, $L$ contains a quadratic extension $\FF_{q^2}$ of $\FF_q$.  By Hasse's bound, we know that $a_\p(\varphi) = \sum_{i=1}^{n/2} a_i T^i$ with $a_i\in \FF_q$.  We need to check that $a_{n/2}\neq 0$.  By \cite{MR2366959}*{Proposition~2.14(i)} or \cite{MR1781330}*{Theorem~5.1},
 \[
 a_{n/2} = \Tr_{\FF_{q^2}/\FF_q}\big( N_{L/\FF_{q^2}}(-\bbar{T}^{q-1})^{-1}\big)=\pm \Tr_{\FF_{q^2}/\FF_q}\big( N_{L/\FF_{q^2}}(\bbar{T}^{q-1})^{-1}\big).
 \]
 Let $\alpha:=N_{L/\FF_{q^2}}(\bbar{T}^{q-1})^{-1} \in \FF_{q^2}^\times$.   Assume that $a_{n/2}=\Tr_{\FF_{q^2}/\FF_q}(\alpha)$ is equal to $0$; we shall derive a contradiction. We have $N_{\FF_{q^2}/\FF_q}(\alpha) = (N_{L/\FF_q}(\bbar{T})^{q-1})^{-1} = 1^{-1} =1$, so $\alpha$ is a root of the polynomial
 \[
 x^2-\Tr_{\FF_{q^2}/\FF_q}(\alpha)x + N_{\FF_{q^2}/\FF_q}(\alpha) = x^2 +1.
 \]
 Since $q\equiv 1 \pmod{4}$, we deduce that $\alpha \in \FF_q^\times$ and hence $\Tr_{\FF_{q^2}/\FF_q}(\alpha) = 2\alpha \neq 0$.
\end{proof}

When $q\equiv 1 \pmod{4}$, the same obstruction to $a_\p(\varphi)=t$ does not occur when we restrict to $\p$ of {odd} degree.   Moreover, there are infinitely many $\p$, with necessarily odd degree, for which $a_\p(\varphi)=0$ (this follows from Theorem~1.1.6 of \cite{MR1185592} with the correction mentioned in \S2 of \cite{MR1606391}).   In fact, it is not unreasonable to expect that $C_{\varphi,t}(n)>0$ whenever $n$ is odd.

What makes Proposition~\ref{P:obstruction} surprising, at least in contrast to the elliptic curve case, is that it cannot be explained by congruence obstructions.   Indeed, the elements $\rho_\varphi(\Frob_\p)$ for $\p$ of \emph{even} degree are dense in $\rho_\varphi\big(G_{\FF_{q^2}(T)}\big) = \GL_2(\widehat{A})$, so for each non-zero ideal $\aA$ of $A$ there are infinitely many $\p$ of even degree for which $a_\p(\varphi)\equiv t\pmod{\aA}$.  The obstruction to $a_\p(\varphi)=t$ for $\p$ of large even degree arises from the infinite place!

In the elliptic curve setting, there is an archimedean component of Lang and Trotter's heuristic which is played by the Sato-Tate law of $E$.   It is responsible for the innocuous factor $2/\pi$ occurring in the constant $C_{E,t}$, and in particular it never gives an obstruction to $a_p(E)=t$ (at least not for large enough $p$).   J.-K.~Yu has proved that there is an analogue of Sato-Tate for Drinfeld modules \cite{MR2018826}, and the obstruction of Proposition~\ref{P:obstruction} can be deduced from it.    The Sato-Tate law has been described in \cite{Zywina-SatoTate}, and the author hopes to state a general Lang-Trotter conjecture in future work.\\

It is hard to give convincing numerical evidence for (\ref{E:LT Drinfeld}), at least compared to that supplied in \S\ref{SS:Koblitz Drinfeld} and \S\ref{SS:cyclic}, since $q^{n/2}/n$ grows much slower than $\pi_A(n) \approx q^n/n$.  For example, some computations show that 
\[
\calP_{\varphi,0}(9)=84,\; \calP_{\varphi,0}(11)=359, \; \calP_{\varphi,1}(9)=62,\; \calP_{\varphi,1}(11)=272,\;  \calP_{\varphi,T}(9)=62,\; \calP_{\varphi,T}(11)=259,
\]
which is compatible with the conjectural asymptotic (\ref{E:LT Drinfeld}) with $M_\varphi=2$ and $C_{\phi,0}(1) \approx 0.5$, $C_{\phi,1}(1) \approx 0.4$, $C_{\phi,T}(1) \approx 0.4$.

\begin{remark}
C.~David and A.C.~Cojocaru proved the upper bound $\calP_{\phi,t}(n) \ll_{\phi} q^{\frac{4}{5}n}/n^{1/5}$ in \cite{MR2441247} (in that paper, they were mainly considering the Drinfeld module analogue of another conjecture of Lang and Trotter).  A key ingredient for their bound was the rank 2 case of Theorem~\ref{T:MT}.    In 
\cite{Zywina-SatoTate}, the Sato-Tate law for $\phi$ is used to prove the upper bound $\calP_{\phi,t}(n) \ll_{\phi,t} q^{\frac{3}{4}n}$.
\end{remark}

\section{The determinant of $\rho_\varphi$}
\label{S:the determinant}
\subsection{The Carlitz module}
\label{SS:Carlitz}

The \defi{Carlitz module} is the Drinfeld module $C \colon A \to F\{\tau\}$ for which $C_T= T +\tau$.  

\begin{prop}[Hayes \cite{MR0330106}] \label{P:Hayes}
For every non-zero ideal $\aA$ of $A$, the representation 
\[
\bbar\rho_{C,\aA} \colon G_F \to \Aut(C[\aA])=(A/\aA)^\times
\]
is surjective.   The representation $\bbar\rho_{C,\aA}$ is unramified at all finite places of $F$ not dividing $\aA$, and for each monic irreducible polynomial $\p$ of $A$ not dividing $\aA$, we have  $\bbar\rho_{C,\aA}(\Frob_\p) \equiv \p \; \bmod{\aA}$.
\end{prop}

In particular, the proposition implies that $\rho_C \colon G_F \to \GL_1(\widehat{A})=\widehat{A}^\times$ is surjective; this gives a rank one example of Theorem~\ref{T:MT}.    

\subsection{The determinant} \label{SS:det}
Let $\p$ be a monic irreducible polynomial of $A$ different from $T$.   For any non-zero ideal $\aA$ of $A$ relatively prime to $\p$,  we know that 
\[
\det(xI-\bbar\rho_{\varphi,\aA}(\Frob_\p)) \equiv x^2 -a_\p(\varphi)x + \epsilon_\p(\varphi)\p \; \pmod{\aA}
\]
for $a_\p(\varphi)\in A$ and $\epsilon_\p(\varphi)\in \FF_q^\times$ that do not depend on $\aA$.  

We can explicitly compute $\epsilon_\p(\varphi)$:  Let $\bbar{T}$ be the image of $T$ in $\FF_\p$.   By Theorem~2.11 of \cite{MR2366959} (with $L=\FF_\p$) we have 
\[
\epsilon_\p(\varphi) = (-1)^{\deg \p} N_{\FF_\p/\FF_q}(-\bbar{T}^{q-1})^{-1} =  (N_{\FF_\p/\FF_q}(\bbar{T})^{q-1})^{-1} = 1,
\]
where the last equality uses that $N_{\FF_\p/\FF_q}(\bbar{T})\neq 0$ since $\p\neq T$.    Thus $(\det\circ \bbar{\rho}_{\varphi,\aA})(\Frob_\p) \equiv \p \equiv \bbar{\rho}_{C,\aA}(\Frob_\p) \bmod{\aA}$.  Using the Chebotarev density theorem, we find that the characters $\det \circ \bbar\rho_{\varphi,\aA} \colon G_F \to (A/\aA)^\times$ and $\bbar\rho_{C,\aA} \colon G_F \to (A/\aA)^\times$ are the same.  

\begin{prop} \label{P:determinant}
The representation $\det \circ \bbar\rho_{\varphi,\aA} \colon G_F \to (A/\aA)^\times$ equals $\bbar\rho_{C,\aA}$, and hence satisfies the properties of Proposition~\ref{P:Hayes}.
\end{prop}

\begin{remark}
That $\det \circ \rho_\varphi = \rho_C$ is not a surprising coincidence.  In the category of $T$-motives it makes sense to take the ``determinant'' of $\varphi$ which gives a rank one Drinfeld $A$-module defined by $T\mapsto T + T^{q-1} \tau$, and this is isomorphic to $C$ over $F$.
\end{remark}

\section{The Drinfeld-Tate uniformization}
\label{S:The Drinfeld-Tate uniformization}

We now fix some notation that will hold for the rest of the section.  Let $\OO$ be a complete discrete valuation ring containing $A$, $\m\subset \OO$ the maximal ideal, $K$ the field of fractions of $\OO$, and $K^\sep$ a separable closure of $K$.    Let $v \colon K^\times \twoheadrightarrow \ZZ$ be the associated discrete valuation (we will also denoted by $v$ the corresponding $\QQ$-valued extension of $v$ to $K^\sep$).   Let $I_K$ be the inertia subgroup of $G_K$ and let $K^\un$ be the maximal unramified extension of $K$ in $K^\sep$.  We will return to our specific Drinfeld module $\varphi$ in \S\ref{SS:Tate example}.

\subsection{Stable reduction}

Let $\phi\colon A \to K\{\tau\}$ be a Drinfeld module of rank $r$.  We say that $\phi$ has \defi{stable reduction} if there exists a Drinfeld module $\phi'\colon A \to \OO\{\tau\}$ such that $\phi'$ and $\phi$ are isomorphic over $K$ and the reduction of $\phi'$ modulo $\m$ is a Drinfeld module (equivalently, the degree of the reduction of $\phi'_T$ is greater that $1$).  We say that $\phi$ has \defi{stable reduction of rank $r_1$} if it has stable reduction and the rank of $\phi'$ modulo $\m$ is $r_1$.  We say that $\phi$ has \defi{good reduction} if it has stable reduction of rank $r$.   Every Drinfeld $A$-module over $K$ has \defi{potentially stable reduction} (i.e., has stable reduction after possibly replacing $K$ by a finite separable extension).

If $\phi\colon A \to K\{\tau\}$ is a Drinfeld module of rank $2$, then the \defi{$j$-invariant} of $\phi$ is defined to be $j_\phi = g^{q+1}/\Delta$ where $\phi_T = T + g \tau + \Delta \tau^2$.  Two Drinfeld $A$-modules over $K$ of rank $2$ have the same $j$-invariant if and only if they are isomorphic over $\bbar{K}$.  The Drinfeld module $\phi$ has {potentially good reduction} if and only if $v(j_\phi)\geq 0$; cf.~\cite{MR2020270}*{Lemma~5.2}.

\subsection{Image of inertia at places of stable bad reduction}

\begin{prop} \label{P:main Tate} 
Let $\phi \colon A \to K\{\tau\}$ be a Drinfeld module of rank $2$ and of generic characteristic that has stable reduction of rank $1$.  Let $\aA$ be a non-zero proper ideal of $A$.  
\begin{romanenum}
\item \label{I:Tate i} There is a basis of $\phi[\aA]$ over $A/\aA$ such that for $\bbar\rho_{\phi,\aA} \colon G_K \to \Aut(\phi[\aA])\cong \GL_2(A/\aA)$ we have
\[
\bbar\rho_{\phi,\aA}(I_K) \subseteq \bigg\{ \left(\begin{array}{cc} 1 & b \\0 & c\end{array}\right) : b\in A/\aA,\, c\in \FF_q^\times \bigg\}.
\]
\item \label{I:Tate ii}  Let $e_\phi$ be the order of $\dfrac{v(j_\phi)}{(q-1)N(\aA)}+\ZZ$ in $\QQ/\ZZ$.
Then $\#\bbar\rho_{\phi,\aA}(I_K)\geq e_\phi.$
\end{romanenum}
\end{prop}

The proof of Proposition~\ref{P:main Tate} will be given in \S\ref{SS:main Tate proof}.

\subsection{Drinfeld-Tate uniformization} \label{SS:DT uniformization}

Let $\psi\colon A \to \OO\{\tau\}$ be a Drinfeld module.  A \defi{$\psi$-lattice} is a finitely generated projective $A$-submodule $\Gamma$ of ${}^\psi\!K^\sep$ that is discrete and is stable under the action of $G_K$.    By discrete we mean that any disk of finite radius in $K^\sep$, with respect to the valuation $v$, contains only finitely many elements of $\Gamma$.  

\begin{defn}  A \defi{Tate datum} over $\OO$ is a pair $(\psi,\Gamma)$, where $\psi$ is a Drinfeld module over $\OO$ and $\Gamma$ is a $\psi$-lattice.  We say that two pairs $(\psi,\Gamma)$ and $(\psi',\Gamma')$ of Tate datum are \defi{isomorphic} if there is an isomorphism from $\psi$ to $\psi'$ such that the corresponding homomorphism ${}^\psi\!K^\sep\to {}^{\psi'}\!K^\sep$ of $A$-modules induces an isomorphism between $\Gamma$ and $\Gamma'$.
\end{defn}

\begin{prop}[Drinfeld \cite{MR0384707}*{\S7}]   \label{P:Tate uniformization}
Let $r_1$ and $r_2$ be positive integers.  There is a natural bijection between the following:
\begin{alphenum}
\item \label{I:Tate uniformization i} 
the set of $K$-isomorphism classes of Drinfeld modules $\phi\colon A \to K\{\tau\}$ of rank $r:=r_1+r_2$ with stable reduction of rank $r_1$;
\item \label{I:Tate uniformization ii} 
the set of $K$-isomorphism classes of Tate datum $(\psi,\Gamma)$ where $\psi\colon A \to \OO\{\tau\}$ is a Drinfeld module of rank $r_1$ with good reduction and $\Gamma$ is a $\psi$-module of rank $r_2$.
\end{alphenum}
\end{prop}

The proposition is not very meaningful as stated; we shall now give a brief description of the implied correspondence.   This correspondence is called the \defi{Drinfeld-Tate uniformization}; see \cite{MR2488548}*{Chapter~4 \S3} for a detailed description and proof.  \\

We start with a Drinfeld module $\psi\colon A \to \OO\{\tau\}$ of rank $r_1$ with good reduction and a $\psi$-lattice $\Gamma$ of rank $r_2$.  Define the power series
\[
e_\Gamma(X) = X \prod_{\gamma \in \Gamma, \, \gamma\neq0} \left( 1 - \frac{X}{\gamma} \right) \in \OO[[X]],
\]
it is $\FF_q$-linear with an infinite radius of convergence and satisfies $e_\Gamma(X)\equiv X \bmod{\m}$; the discreteness of $\Gamma$ is key here.  We may then view $e_\Gamma$ as an element of $\OO\{\!\{\tau\}\!\}$; the (non-commutative) ring of formal power series in $\tau$ with coefficients in $\OO$.  There exists a unique Drinfeld $A$-module $\phi$ over $\OO$ such that $e_\Gamma \psi_a=\phi_a e_\Gamma$ holds for all $a\in A$.    This is the desired Drinfeld module $\phi$; it has rank $r_1+r_2$ with stable reduction of rank $r_1$.  That $\phi$ has stable reduction of rank $r_1$ is clear since $\phi_T\equiv \phi_T e_\Gamma = e_\Gamma \psi_T \equiv \psi_T \bmod{\m}$ and $\psi$ has good reduction.\\

In the other direction, start with a Drinfeld $A$-module $\phi$ of rank $r:=r_1+r_2$  over $K$ which has stable reduction of rank $r_1$.    After possibly replacing $\phi$ with a $K$-isomorphic Drinfeld module, we may assume that $\phi$ takes values in $\OO\{\tau\}$.   There exists a unique Drinfeld module $\psi\colon  A\to \OO\{\tau\}$ of of rank $r_1$ and a unique element
$u= \tau^0 + \sum_{i=1}^\infty a_i \tau^i \in \OO\{\!\{\tau\}\!\}$
with $a_i\in \m$ and $|a_i|\to 0$, such that 
\begin{equation} \label{E:Tate Uniformization 1}
u \psi_a=\phi_au
\end{equation}
for all $a\in A$.  Drinfeld shows that $u$ defines an analytic homomorphism.   Let $\Gamma$ be the kernel of $u$.   It is a subgroup of $K^\sep$, and moreover it is a $\psi$-lattice of rank $r_2$.  The pair $(\psi,\Gamma)$ is the desired Tate uniformization of $\phi$.\\

Fix an $a \in A - \FF_q$.  In the proof that $\Gamma$ is a lattice, one makes use of the following $G_K$-equivariant short exact sequence of $A$-modules:  
\begin{equation} \label{E:SES Tate}
1 \to \psi[a]=\psi_a^{-1}(0) \to \psi_a^{-1}(\Gamma)/\Gamma \xrightarrow{\psi_a} \Gamma/a\Gamma \to 1.
\end{equation}
We also have an isomorphism
\begin{equation} \label{E:Tate uniformization torsion}
\psi_a^{-1}(\Gamma)/\Gamma \xrightarrow{\sim} \phi[a], \quad z + \Gamma \mapsto u(z)
\end{equation}
of $A[G_K]$-modules  (it is a well-defined map by (\ref{E:Tate Uniformization 1})).

\subsection{Proof of Proposition~\ref{P:main Tate}}  \label{SS:main Tate proof}
Fix a Drinfeld module $\phi \colon A \to \OO\{\tau\}$ of rank $2$ that has stable reduction of rank $1$.   Let $(\psi, \Gamma)$ be the corresponding Tate uniformization as in \S\ref{SS:DT uniformization}.   We have $\aA=(a)$ for some $a\in A$. Using the isomorphism (\ref{E:Tate uniformization torsion}), it suffices to prove the analogous statement of the proposition for $\psi_a^{-1}(\Gamma)/\Gamma$.  We first consider the Galois action on the pieces $\psi[\aA]$ and $\Gamma/a\Gamma$.\\

The Drinfeld module $\psi \colon A \to \OO\{\tau\}$ has rank $1$ and good reduction, so the Galois representation $\bbar\rho_{\psi,\aA}\colon G_K \to \Aut(\psi[\aA])=(A/\aA)^\times$ is unramified.  Choose a generator $w$ of $\psi[\aA]$ as an $A/\aA$-module; thus $\sigma(w)=w$ for all $\sigma \in I_K$.\\

The lattice $\Gamma$ is a free $A$-module of rank $1$.  Fix a generator $\gamma$ of $\Gamma$, it is well-defined up to multiplication by an element of $\FF_q^\times$.   Since the lattice $\Gamma$ is stable under the Galois action, there is a character $\chi_\Gamma \colon G_K \to \FF_q^\times$  such that $\sigma(\gamma) = \chi_\Gamma(\sigma)\gamma$ for all $\sigma\in G_K.$  

Choose a $z\in K^\sep$ for which $\psi_a(z)=\gamma$; this is equivalent to choosing a splitting of the short exact sequence (\ref{E:SES Tate}) of $A/\aA$-modules.    For any $\sigma \in I_K$,
\[
\psi_a(\sigma(z))=\sigma(\psi_a(z))=\sigma(\gamma) =\chi_\Gamma(\sigma)\gamma = \chi_\Gamma(\sigma)\psi_a(z) = \psi_a(\chi_\Gamma(\sigma) z).
\]
Thus $\sigma(z)- \chi_\Gamma(\sigma) z \in \psi[\aA]$, hence there exists a unique $b_\sigma\in A/\aA$ such that
\[
\sigma(z)=\chi_\Gamma(\sigma) z + b_\sigma w.  
\]
Thus with respect to the basis $\{w+\Gamma,z+\Gamma\}$ of $\psi_a^{-1}(\Gamma)/\Gamma$, an automorphism $\sigma \in I_K$ acts via the matrix 
\begin{equation*} \label{E:Tate matrix}
\left(\begin{array}{cc}1 & b_\sigma \\0 & \chi_\Gamma(\sigma)\end{array}\right).
\end{equation*}
This proves part (\ref{I:Tate i}). 

If $v(z) \geq 0$, then $v(\gamma)=v(\psi_a(z))\geq 0$ since $\psi_a$ has coefficients in $\OO$.  However the discreteness of the lattice $\Gamma$ implies that $v(\gamma)<0$, so we must have $v(z)<0$.  Therefore,
\[
v(\gamma) = v(\psi_a(z))=v(z^{q^{\deg a}}) = q^{\deg a} v(z) = N(\aA) v(z).
\]
Let $K'$ be the smallest extension of $K^\un$ in $K^\sep$ for which $\Gal(K^\sep/K')$ acts trivially on $\psi_a^{-1}(\Gamma)/\Gamma$.   The field $K'$ is of course equal to $K^\un(\phi[\aA])$, and $\bbar\rho_{\phi,\aA}(I_K)\cong \Gal(K'/K^\un)$.  Since $\psi[\aA]\subseteq K^\un$, we find that $K'=K^\un(z)$.  The ramification index of the extension $K^\un(z)/K^\un$ is at least the order of $v(z)+\ZZ$ in $\QQ/\ZZ.$
By \cite{MR2020270}*{Lemma~5.3}, we have $v(\gamma) = v(j_\phi)/(q-1)$ and thus
\[
v(z) = \frac{v(\gamma)}{N(\aA)} =  \frac{v(j_\phi)}{(q-1)N(\aA)}.
\]
Part (\ref{I:Tate ii}) now follows immediately.  

\subsection{Our example} \label{SS:Tate example}
We will now apply the above theory to our specific Drinfeld module $\varphi \colon A \to F\{\tau\}$ with  $\varphi_T=T + \tau - T^{q-1}\tau^2.$   Let $p$ be the prime dividing $q$.

\begin{prop} \label{P:Tate example} 
Let $I_T$ be an inertia subgroup of $G_F$ at $T$.
For any non-zero ideal $\aA$ of $A$, $\bbar\rho_{\phi,\aA}(I_T)$ is a $p$-Sylow subgroup of $\Aut(\phi[\aA])\cong \GL_2(A/\aA).$   Equivalently, $\#\bbar\rho_{\phi,\aA}(I_T) = N(\aA).$
\end{prop}
\begin{proof}
The Drinfeld module $\varphi$ has stable reduction of rank $1$ at $(T)$.  The field $K:= \FF_q(\!(T)\!)$ is the completion of $F$ with respect to $T$.  Let $v_T$ be the corresponding valuation normalized so that $v_T(T)=1$. 

We know from Proposition~\ref{P:determinant} that $\det\circ \bbar\rho_{\varphi,\aA} = \bbar\rho_{C,\aA}$.  Since $C$ has good reduction at $(T)$, we must have $\det(\bbar\rho_{\varphi,\aA}(I_K))=\bbar\rho_{C,\aA}(I_K)=1$.  This combined with Proposition~\ref{P:main Tate}(\ref{I:Tate i}) shows that $\bbar\rho_{\varphi,\aA}(I_K)$ is contained in a subgroup of $\GL_2(A/\aA)$ of order $N(\aA)$.   By Proposition~\ref{P:main Tate}(\ref{I:Tate ii}), $v_T(j_\phi)= -(q-1)$ implies that $\#\bbar\rho_{\varphi,\aA}(I_T) \geq N(\aA)$.
\end{proof}

\section{Irreducibility} \label{S:irreducibility}

\begin{prop} \label{P:irreducibility}  
The $\FF_\lambda[G_F]$-module $\varphi[\lambda]$ is irreducible for every finite place $\lambda$ of $F$.
\end{prop}

We now suppose that $\varphi[\lambda]$ is a reducible $\FF_\lambda[G_F]$-module for a fixed $\lambda$.   We shall eventually obtain a contradiction and thus prove Proposition~\ref{P:irreducibility}.   The strategy of this section is based on \S5.4 of \cite{MR0387283}.  By choosing an appropriate basis of $\varphi[\lambda]$, we may assume that the image of $\bbar\rho_{\varphi,\lambda} \colon G_F \to \Aut(\varphi[\lambda])\cong \GL_2(\FF_\lambda)$ lies in the group of upper triangular matrices.  Moreover, there are two characters $\chi$ and $\chi' :G_F \to \FF_\lambda^\times$ such that $\bbar\rho_\lambda$ is represented in matrix form by $\left(\begin{array}{cc}\chi & * \\0 & \chi'\end{array}\right)$.  We will now try to determine these characters.

\begin{lemma} \label{L:unramified character}
The characters $\chi$ and $\chi'$ are unramified at all finite places $\p\neq \lambda$.  One of these two characters is unramified at all the finite places of $F$. 
\end{lemma}
\begin{proof}
First consider the place $\p=(T)$.  By Proposition~\ref{P:Tate example}, the order of every element of $\bbar\rho_{\varphi,\lambda}(I_{\p})$ divides some power of $q$ (where $I_\p$ is the inertia subgroup of $G_F$ at $\p$).  Therefore, $\chi(I_{\p})=1$ and $\chi'(I_{\p})=1$ since both take values in a group of cardinality relatively prime to $q$.  

Now consider a finite place $\p$ not equal to $\lambda$ or $(T)$.  Since $\varphi$ has good reduction at $\p$, we find that $\bbar\rho_{\varphi,\lambda}$ is unramified at $\p$ and hence so are $\chi'$ and $\chi''$.  

Finally consider the case where $\p=\lambda$ and $\p\neq (T)$.   The reduction of $\varphi$ modulo $\p$ has height $1$ (if it had height $2$, then \cite{MR2499411}*{Proposition~2.7(ii)} would imply that $\varphi[\lambda]$ is an irreducible $G_F$-module).   By \cite{MR2499411}*{Proposition~2.7}, $\bbar\rho_{\varphi,\lambda}(I_\p)$ acts on $\varphi[\lambda]$ via matrices of the form $\left(\begin{matrix} * & * \\ 0 & 1 \end{matrix}\right)$ with respect to an appropriate basis.  Hence $\chi(I_\p)=1$ or $\chi'(I_p)=1$.
\end{proof}

\begin{lemma} \label{L:trivial character}
One of the characters $\chi,\chi'\colon G_F \to \FF_\lambda^\times$ is of the form $G_F \twoheadrightarrow \Gal(\FFbar_q/\FF_q) \to \FF_\lambda^\times$, where the first map is restriction.
\end{lemma}
\begin{proof}
By Lemma~\ref{L:unramified character}, one of the characters $\chi$ or $\chi' \colon G_F\to \FF_\lambda^\times$, without loss of generality $\chi'$, is unramifed at all finite places of $F$.     Thus we may view $\chi'$ as a $\FF_\lambda^\times$-valued character of the \'etale fundamental group of $\AA^1_{\FF_q}$.  Since $\AA^1_{\FFbar_q}$ has no non-trivial \'etale covers of order prime to $q$, we deduce that $\chi'\colon G_F\to \FF_\lambda^\times$ is trivial on $\Gal(F^\sep/\FFbar_q(T))$.  The lemma is now immediate.
\end{proof}

We can now express the values $a_\p(\varphi) \bmod{\lambda}$ in terms of the characters $\chi$ and $\chi'$.

\begin{lemma} \label{L: semistable congruence}
Let  $\lambda$ be a finite place of $F$ for which $\varphi[\lambda]$ is a reducible $\FF_\lambda[G_F]$-module.  There is a $\zeta\in \FF_\lambda^\times$ such that for any monic irreducible polynomial $\p\in A$ that is not $T$ or $\lambda,$ we have
\begin{equation} \label{E: semistable congruence}
a_\p(\varphi) \equiv  \zeta^{-\deg{\p}}\, \p + \zeta^{\deg{\p}} \;\; \bmod{\lambda}.
\end{equation}
\end{lemma}
\begin{proof}
By Lemma~\ref{L:trivial character}, one of the characters $\chi, \chi'\colon G_F\to \FF_\lambda^\times$, say $\chi'$, factors through a character $\Gal(\FFbar_q/\FF_q) \to \FF_\lambda^\times$.   Hence there is a $\zeta\in\FF_\lambda^\times$ such that $\chi'(\Frob_\p)=\zeta^{\deg \p}$ for any monic irreducible polynomial $\p$ that is not $T$ or $\lambda$.    By Proposition~\ref{P:determinant}, we know that $\chi(\Frob_\p)\chi'(\Frob_\p) = \det(\bbar\rho_{\varphi,\lambda}(\Frob_\p)) \equiv \p \bmod{\lambda}$, hence
\[
\chi(\Frob_\p) \equiv \zeta^{-\deg{\p}}\, \p \;\;\bmod{\lambda} \quad \text{ and }\quad 
\chi'(\Frob_\p) = \zeta^{\deg \p} \;\; \bmod{\lambda}.
\]
We deduce that
\begin{align*}
a_\p(\varphi) \equiv \tr \big(\bbar\rho_{\varphi,\lambda}(\Frob_\p)\big) & = \chi(\Frob_\p) +  \chi'(\Frob_\p) \equiv  \zeta^{-\deg{\p}}\, \p + \zeta^{\deg{\p}} \;\; \bmod{\lambda}.  \qedhere
\end{align*}
\end{proof}

By checking (\ref{E: semistable congruence}) for various primes $\p$, we will be able to rule out many $\lambda$; it turns out that we will only need to consider $\p$ of degree $1$.

\begin{lemma} \label{L:traces for degree 1 primes}
Let $\p$ be the irreducible polynomial $T - c \in A$ with $c\in \FF_q^\times$.  Then $a_\p(\varphi)=1$.
\end{lemma}
\begin{proof}
The image of $T$ in $\FF_\p$ is $c$.  By Proposition~2.14(ii) of \cite{MR2366959} (with $L=\FF_\p=\FF_q$) we have $a_\p(\varphi) = - (-1/c^{q-1}) = 1$.
\end{proof}

Since $q\geq 5$, there exist $c_1,c_2\in \FF_q^\times$ such that $\lambda$, $T-c_1$ and $T-c_2$ are distinct.  By Lemma~\ref{L: semistable congruence} and \ref{L:traces for degree 1 primes} with $\p=T-c_i$, we get
\[
1 \equiv   \zeta^{-1}(T-c_i)+ \zeta\;\; \bmod{\lambda}
\]
This implies that $T\equiv \zeta - \zeta^2 + c_i  \pmod{\lambda}$ for distinct $c_1,c_2\in \FF_q\subseteq \FF_\lambda$; this is our contradiction.

\section{Proof of Theorem~\ref{T:MT example}}
\label{S:proof}
In different respects, Propositions~\ref{P:determinant}, \ref{P:Tate example}  and \ref{P:irreducibility} all show that the group $\rho_{\varphi}(G_F)$ is large.  We now combine everything together to prove that indeed $\rho_{\varphi}(G_F)=\GL_2(\widehat{A})$.   This will require some extra group theory which we have collected in Appendix~\ref{S:group theory}.\\

Let $F^\ab$ be the maximal abelian extension of $F$ in $F^\sep$.    Note that $\bbar{\rho}_{\varphi,\aA}(G_{F^\ab})\subseteq \SL_2(A/\aA)$ for each non-zero ideal $\aA$ of $A$.  We first show that the $\lambda$-adic representations of $\varphi$ are surjective.

\begin{lemma} \label{L:base camp} 
For every finite place $\lambda$ of $F$, we have $\rho_{\varphi,\lambda}(G_F) = \GL_2(A_\lambda)$ and $\rho_{\varphi,\lambda}(G_{F^\ab}) = \SL_2(A_\lambda).$
\end{lemma}
\begin{proof}
By Propositions~\ref{P:irreducibility} and \ref{P:Tate example}, the group $\bbar{\rho}_{\varphi,\lambda}(G_F) \subseteq \GL_2(\FF_\lambda)$ acts irreducibly on $\varphi[\lambda]\cong \FF_\lambda^2$ as an $\FF_\lambda$-module  and it also contains a group of order $N(\lambda)$.    Lemma~\ref{L:easy group theory} then implies that $\bbar{\rho}_{\varphi,\lambda}(G_F) \supseteq \SL_2(\FF_\lambda)$.   We have $\bbar{\rho}_{\varphi,\lambda}(G_F) = \GL_2(\FF_\lambda)$ since $\det( \bbar{\rho}_{\varphi,\lambda} (G_F) )=\FF_\lambda^\times$ by Proposition~\ref{P:determinant}.

The group $H:=\rho_{\varphi,\lambda}(G_F)$ is closed in $\GL_2(A_\lambda)$, and satisfies $\det(H) = A_\lambda^\times$ by Proposition~\ref{P:determinant}.     The group $H\bmod{\lambda^2}=\bbar{\rho}_{\varphi,\lambda^2}(G_F)$ contains a non-scalar matrix that is congruent to the identity modulo $\lambda$ by Proposition~\ref{P:Tate example}.   We just verified that $H \bmod{\lambda} =\bbar\rho_{\varphi,\lambda}(G_F)= \GL_2(\FF_\lambda)$.   Applying Lemma~\ref{L:Pink group theory}, we deduce that $H = \GL_2(A_\lambda).$   The group $\rho_{\varphi,\lambda}(G_{F^\ab})$ is just the commutator subgroup of $H=\GL_2(A_\lambda)$ which from Lemma~\ref{L:SL2 commutator} is $\SL_2(A_\lambda)$.
\end{proof}

Having surjective representations $\rho_{\varphi,\lambda}$ is \emph{not} enough to deduce that $\rho_\varphi$ is surjective.   There may be interdependencies between the representations.   We now show that the mod $\lambda$ representations are pairwise independent.

\begin{lemma} \label{L:independence mod}
Let $\lambda_1$ and $\lambda_2$ be distinct finite places of $F$, and let $\aA=\lambda_1\lambda_2$ be the corresponding ideal of $A$.  Then 
$\bbar\rho_{\varphi,\aA} (G_F) = \GL_2(A/\aA)$ and $\bbar\rho_{\varphi,\aA} (G_{F^\ab}) = \SL_2(A/\aA).$
\end{lemma}
\begin{proof}
Define $H:= \bbar\rho_{\varphi, \aA} (G_{F})$ and $H':= H \cap \SL_2(A/\aA).$  We shall verify the three conditions of Lemma~\ref{L:Ribet}, which will imply that $\bbar\rho_{\varphi, \aA} (G_{F})=\GL_2(A/\aA)$.   We will then have $\bbar\rho_{\varphi,\aA} (G_{F^\ab}) = \SL_2(A/\aA)$ automatically since $\SL_2(A/\aA)$ is the commutator subgroup of $\GL_2(A/\aA)$ by Lemma~\ref{L:SL2 commutator}.

Condition (\ref{I:Ribet i}) of Lemma~\ref{L:Ribet} follows from Proposition~\ref{P:determinant}.  By Lemma~\ref{L:base camp}  we have $\bbar\rho_{\varphi,\lambda_i}(G_{F^\ab})=\SL_2(\FF_{\lambda_i})$, so condition (\ref{I:Ribet ii}) follows since $\bbar\rho_{\varphi,\aA}(G_{F^\ab})\subseteq H'$.  

Take any $c\in \FF_q^\times$ such that $\p=T-c$ is not $\lambda_1$ or $\lambda_2$.   By Lemma~\ref{L:traces for degree 1 primes},  we have
\[
\det\big(\bbar\rho_{\varphi,\aA}(\Frob_\p)\big)/\tr\big( \bbar\rho_{\varphi,\aA}(\Frob_\p) \big)^2 \equiv \p/a_\p(\varphi)^2 = \p  = T- c \; \pmod{\aA}.
\]
One readily checks that the subring of $A/\aA$ generated by the $T-c$, with at most two of the $c\in \FF_q^\times$ excluded, is all of $A/\aA$.   This verifies condition (\ref{I:Ribet iii}) of Lemma~\ref{L:Ribet}.
\end{proof}

\begin{lemma} \label{L:adic independence}
Let $\lambda_1$ and $\lambda_2$ be distinct finite places of $F$.  Define 
\begin{align*}
\rho \colon G_F &\to \GL_2(A_{\lambda_1}) \times \GL_2(A_{\lambda_2}), \quad
\sigma  \mapsto (\rho_{\varphi,\lambda_1}(\sigma), \rho_{\varphi,\lambda_2}(\sigma)).
\end{align*}
Then $\rho(G_{F^\ab})=\SL_2(A_{\lambda_1}) \times \SL_2(A_{\lambda_2})$ and $\rho(G_{F})=\GL_2(A_{\lambda_1}) \times \GL_2(A_{\lambda_2})$.
\end{lemma}
\begin{proof}
To prove that $\rho(G_{F^\ab})=\SL_2(A_{\lambda_1}) \times \SL_2(A_{\lambda_2})$, it suffices to show that for any positive integers $n_1$ and $n_2$, we have 
\[
\bbar\rho_{\varphi, \aA}(G_{F^\ab}) =\SL_2(A/\aA) 
\]
where $\aA=\lambda_1^{n_1} \lambda_2^{n_2}$.    That $\rho$ is surjective will follow from this and Proposition~\ref{P:determinant}.

Suppose that $H:=\bbar\rho_{\varphi, \aA}(G_{F^\ab})$ is not equal to  $\SL_2(A/\aA)= \SL_2(A/\lambda_1^{n_1})\times \SL_2(A/\lambda_2^{n_2})$.    Let $N_1$ and $N_2$ be the kernels of the projections $H\to \SL_2(A/\lambda_2^{n_2})$ and $H\to \SL_2(A/\lambda_1^{n_1})$, respectively.   Each of these projections are surjective by Lemma~\ref{L:base camp}.  By Lemma~\ref{L:Goursat} we may view $N_i$ as a normal subgroup of $\SL_2(A/\lambda_i^{n_i})$ and the image of $H$ in $\SL_2(A/\lambda_1^{n_1})/N_1 \times \SL_2(A/\lambda_2^{n_2})/N_2$ is the graph of an isomorphism $\SL_2(A/\lambda_1^{n_1})/N_1 \xrightarrow{\sim} \SL_2(A/\lambda_2^{n_2})/N_2$.

By our assumption on $H$, the groups $\SL_2(A/\lambda_i^{n_1})/N_i$ are non-trivial.  So by Lemma~\ref{L:SL2 commutator}, $N_i$ is a subgroup of the group of $B\in \SL_2(A/\lambda_i^{n_1})$ with $B\equiv \pm I \bmod{\lambda}$.   Therefore the image of $H$ (equivalently, the image of $\bbar\rho_{\varphi,\lambda_1\lambda_2}(G_{F^\ab})$) in
\[
\SL_2(\FF_{\lambda_1})/\{\pm I\} \times \SL_2(\FF_{\lambda_2}) /\{\pm I\}
\]
is the graph of an isomorphism $\SL_2(\FF_{\lambda_1})/\{\pm I\} \xrightarrow{\sim} \SL_2(\FF_{\lambda_2}) /\{\pm I\}$.   However, this contradicts Lemma~\ref{L:independence mod} which says that $\bbar\rho_{\varphi,\lambda_1\lambda_2}(G_{F^\ab}) = \SL_2(\FF_{\lambda_1}) \times \SL_2(\FF_{\lambda_1}).$  Therefore, $\bbar\rho_{\varphi, \aA}(G_{F^\ab})= \SL_2(A/\aA)$.
\end{proof}

We can now finish the proof of Theorem~\ref{T:MT example}.   We first show that  $\rho_\varphi(G_{F})=\GL_2(\widehat{A})$.   Again by Proposition~\ref{P:determinant} we have $\det(\rho_\varphi(G_F))=\widehat{A}^\times$, so it suffices to show that $\rho_\varphi(G_{F^\ab})=\SL_2(\widehat{A}).$  The equality $\rho_\varphi(G_{F^\ab})= \SL_2(\widehat{A})$ is equivalent to having 
\[
\bbar{\rho}_{\varphi,\aA}(G_{F^\ab})= \SL_2(A/\aA) = \prod_{\lambda^n \parallel \aA} \SL_2(A/\lambda^n)
\] 
for every non-zero ideal $\aA$ of $A$.    By Lemma~\ref{L:SL2 commutator}, the groups $\SL_2(A/\lambda^n)$ have no abelian quotients.   Therefore by Lemma~\ref{L:Goursat cor}, we need only show that $\bbar\rho_{\varphi,\aA} (G_{F^\ab}) = \SL_2(A/\aA)$ for $\aA$ of the form $\lambda_1^{n_1}\lambda_2^{n_2}$ where $\lambda_1$ and $\lambda_2$ are distinct maximal ideals of $A$, and $n_1$ and $n_2$ are positive integers.  This is immediate from Lemma~\ref{L:adic independence}.

Finally, we show that  $\rho_\varphi\big(G_{\FFbar_q(T)}\big)=\GL_2(\widehat{A})$.   Since $\FFbar_q(T)/\FF_q(T)$ is an abelian extension and the commutator subgroup of $\GL_2(\widehat{A})$ is $\SL_2(\widehat{A})$, it suffices to show that $(\det \circ \rho_\varphi)(G_{\FFbar_q(T)}\big)=\widehat{A}^\times.$  Again this is easily verified using the description of of $\det \circ \rho_\varphi$ in \S\ref{SS:Carlitz}.

\subsection*{Acknowledgements}
The calculations in \S\ref{S:conjectures} were performed using Magma \cite{Magma}.

\appendix
\section{Group theory}
\label{S:group theory}

In this appendix we collect all the group theory needed in \S\ref{S:proof} to prove our theorem.  The point of that section was to show that certain closed subgroups of $\GL_2(\widehat{A})$ and $\SL_2(\widehat{A})$ (i.e., $\rho_{\varphi}(G_F)$ and $\rho_\varphi(G_{F^\ab}),$ respectively) were the full groups.   Note that the material in this section makes no reference to Drinfeld modules, though it will use our ongoing assumption that $A=\FF_q[T]$ with $q\geq 5$ an odd prime power.

We start with the following easy generalization of \cite{MR0263823}*{IV-20 Lemma~2}.
\begin{lemma} \label{L:easy group theory}
Let $\FF$ be a finite field.   Let $H$ be a subgroup of $\GL_2(\FF)$ such that:
\begin{itemize}
\item $H$ contains a subgroup of order $\#\FF$; 
\item the $\FF[H]$-module $\FF^2=\FF\times \FF$ is irreducible.
\end{itemize}
Then $H$ contains $\SL_2(\FF)$.
\end{lemma}
\begin{proof}
Let $P_1$ be a subgroup of $H$ of order $\#\FF=p^s$; it is a $p$-Sylow subgroup of $\GL_2(\FF)$ and hence also of $H$.   There is a unique one dimensional $\FF$-subspace $W_1$ of $\FF^2$ that is fixed by every element of $P_1$.

If $P_1$ is a normal subgroup of $H$, then one finds that $W_1$ is stable under the action of $H,$ which would contradict our irreducibility assumption.  Therefore, there is a second subgroup $P_2\neq P_1$ of $H$ with cardinality $\#\FF$.  Let $W_2$ be the unique one dimensional $\FF$-subspace of $\FF^2$ that is fixed by every element of $P_2$.

With respect to a basis $\{w_1,w_2\}$ of $\FF^2$ with $w_1\in W_1$ and $w_2 \in W_2$, the subgroups $P_1$ and $P_2$ of $H$ become
\[
\Big\{ \left(\begin{array}{cc}1 & x \\0 & 1\end{array}\right) : x \in \FF \Big\} \quad \text{ and }\quad 
\Big\{ \left(\begin{array}{cc}1 & 0 \\x & 1\end{array}\right) : x \in \FF \Big\}
\]
respectively.  Now take any matrix $M=\left(\begin{smallmatrix}A & B \\C & D\end{smallmatrix}\right) \in \SL_2(\FF)$.   First suppose that $B\neq 0$.  For $a,b,c\in \FF$, we have
\[
\left(\begin{array}{cc}1 & 0 \\a & 1\end{array}\right)
\left(\begin{array}{cc}1 & b \\0 & 1\end{array}\right)
\left(\begin{array}{cc}1 & 0 \\c & 1\end{array}\right)
=\left(\begin{array}{cc}1+bc & b \\a+c+abc & 1+ab\end{array}\right).
\]
So setting $b=B$ and solving $1+bc=A$ and $1+ab=D$ for $a$ and $c$ (recall that $B\neq0$), we find an expression for $M$ as a product of matrices in $P_1$ and $P_2$ (that $a+c+abc=C$ is automatic since our matrices have determinant $1$ and $b=B\neq 0$).  Therefore $M\in H$.   An analogous argument shows that $M\in H$ when $C\neq 0$.   Finally in the case $B=C=0$, we simply note that 
$\left(\begin{smallmatrix}-1 & 0 \\0 & -1\end{smallmatrix}\right)=\left(\begin{smallmatrix}1 & 0 \\1 & 1\end{smallmatrix}\right)\left(\begin{smallmatrix}1 & -2 \\0 & 1\end{smallmatrix}\right)\left(\begin{smallmatrix}1 & 0 \\1 & 1\end{smallmatrix}\right)\left(\begin{smallmatrix}1 & -2 \\0 & 1\end{smallmatrix}\right) \in H.$
\end{proof}

The following two lemmas give some useful results about $\GL_2(A_\lambda)$ and $\SL_2(A_\lambda)$.

\begin{lemma} \label{L:Pink group theory}
Let $\lambda$ be a finite place of $F$, and let $H$ be a closed subgroup of $\GL_2(A_\lambda)$.  Suppose that $\det(H)=A_\lambda^\times$, $H \bmod{\lambda} = \GL_2(\FF_\lambda)$, and $H \bmod{\lambda^2}$ contains a non-scalar matrix that is congruent to the identity mod $\lambda$.  Then $H=\GL_2(A_\lambda)$.
\end{lemma}
\begin{proof}
This is  Proposition~4.1 of \cite{MR2499412} (note that $N(\lambda)\geq q\geq 5$).
\end{proof}

\begin{lemma} \label{L:SL2 commutator}
For each finite place $\lambda$ of $F$, the group $\SL_2(A_\lambda)$ is its own commutator subgroup.
The only normal subgroup of $\SL_2(A_\lambda)$ with simple quotient is the group consisting of the $B \in \SL_2(A_\lambda)$ for which $B\equiv \pm I \bmod{\lambda}$.
\end{lemma}
\begin{proof}
We first prove that $\SL_2(A_\lambda)$ is its own commutator subgroup.  Let $H$ be the commutator subgroup of $\SL_2(A_\lambda)$. It is a closed normal subgroup of $\SL_2(A_\lambda)$ and $\GL_2(A_\lambda)$.    Define $S^0:=\SL_2(A_\lambda)$, and for each $i\geq 1$ we let $S^i$ be the group of $s\in \SL_2(A_\lambda)$ with $s\equiv 1 \bmod{\lambda^i}$.   For $i\geq 0$, define $H^i := H \cap S^i$.     

For $i\geq 0$, we define $S^{[i]}:= S^i/S^{i+1}$ and $H^{[i]}:= H^i/H^{i+1}$.  There is a natural inclusion $H^{[i]} \hookrightarrow S^{[i]},$ and it suffices to show that $H^{[i]} = S^{[i]}$ for all $i\geq 0$.

Reduction modulo $\lambda$ induces an isomorphism $S^{[0]} \xrightarrow{\sim} \SL_2(\FF_\lambda)$  with the image of $H^{[0]}$ being  the commutator subgroup of $\SL_2(\FF_\lambda)$.    Since $\SL_2(\FF_\lambda)/\{\pm I\}$ is simple, we find that $[\SL_2(\FF_\lambda): H^{[0]}]=1$ or $2$.   Since $\SL_2(\FF_\lambda)$ is generated by elements of order $p$ (use Lemma~\ref{L:easy group theory}), it has no normal subgroup of index $2$.   Therefore, $H^{[0]} = S^{[0]}.$

Now fix an $i\geq 1$.  Let $\sl_2(\FF_\lambda)$ be the (additive) group of matrices in $M_2(\FF_\lambda)$ with trace $0$.  We have an isomorphism
\begin{equation} \label{E: sl2 basic}
S^{[i]} \xrightarrow{\sim} \sl_2(\FF_\lambda), \quad \big[ 1 + \lambda^i y ] \mapsto [y];
\end{equation}
where we are now viewing $\lambda$ as a monic polynomial.   Conjugation by $\GL_2(A_\lambda)$ acts on both sides of (\ref{E: sl2 basic}), and it factors through conjugation by  $\GL_2(\FF_\lambda)$.   By \cite{MR2499412}*{Proposition~2.1}, $\sl_2(\FF_\lambda)$ is an irreducible $\GL_2(\FF_\lambda)$-module (this uses that $q$ is odd).   Now consider $H^{[i]} \hookrightarrow S^{[i]}$.  Since $H$ is a normal subgroup of $\GL_2(A_\lambda)$, we find that $H^{[i]}$ is stable under the $\GL_2(\FF_\lambda)$-action.   So we need only prove that $H^{[i]}\neq 1$.

Consider the commutator map $S^{0}\times S^{i} \to S^{i},$  $(g,h)\mapsto ghg^{-1}h^{-1}$.  This induces a map $S^{[0]}\times S^{[i]} \to S^{[i]}$ that takes values in $H^{[i]}$, and by using the identification $S^{[0]}=\SL_2(\FF_\lambda)$ and (\ref{E: sl2 basic}) it becomes
\[
\SL_2(\FF_\lambda) \times \sl_2(\FF_\lambda) \to \sl_2(\FF_\lambda),\quad (s,X) \mapsto sXs^{-1} - X.
\]
This map is non-zero, so $H^{[i]}\neq 1$.   Therefore, $H=\SL_2(A_\lambda)$.

Now let $N$ be a normal subgroup of $\SL_2(A_\lambda)$ for which $\SL_2(A_\lambda)/N$ is simple.   Since every $p$-group is solvable, the Jordan-H\"older factors of $\SL_2(A_\lambda)$ are $\SL_2(\FF_\lambda)/\{\pm I\}$,  $\ZZ/2\ZZ$ and $\ZZ/p\ZZ$.   We have just shown that $\SL_2(A_\lambda)$ has no abelian quotients, so $\SL_2(A_\lambda)/N\cong \SL_2(\FF_\lambda)/\{\pm I\}$.   Let $N'$ be the group consisting of $B\in \SL_2(A_\lambda)$ with $B\equiv \pm I \bmod{\lambda}$, it is also a normal subgroup of $\SL_2(A_\lambda)$ with quotient isomorphic to $\SL_2(\FF_\lambda)/\{\pm I\}$.   We must have $N\subseteq N'$, otherwise $NN'/N'$ would be a non-trivial normal subgroup of $\SL_2(A_\lambda)/N'$.  Similarly, $N'\subseteq N$.
\end{proof}

\begin{lemma}[Goursat's lemma \cite{MR0457455}*{Lemma~5.2.1}]  \label{L:Goursat}
Let $B_1$ and $B_2$ be finite groups and suppose that $H$ is a subgroup of $B_1\times B_2$ for which the two projections $p_1\colon H \to B_1$ and $p_2\colon H \to B_2$ are surjective.  Let $N_1$ be the kernel of $p_2$ and let $N_2$ be the kernel of $p_1$.  We may view $N_1$ as a normal subgroup of $B_1$ and $N_2$ as a normal subgroup of $B_2$.  Then the image of $H$ in $B_1/N_1 \times B_2/N_2$ is the graph of an isomorphism $B_1/N_1 \xrightarrow{\sim} B_2/N_2$.
\end{lemma}

\begin{remark}
In the setting of the above lemma, we will have $H=B_1\times B_2$ if and only if $N_1=B_1$ and $N_2=B_2$.
\end{remark}

\begin{lemma}[\cite{MR0457455}*{Lemma~5.2.2}]  \label{L:Goursat cor} Let $S_1, S_2, \cdots, S_k$ be finite groups with no non-trivial abelian quotients.  Let $H$ be a subgroup of $S_1\times \cdots \times S_k$ such that each projection $H\to S_i\times S_j$ ($1\leq i < j \leq k$) is surjective.  Then $H=S_1\times \cdots \times S_k.$
\end{lemma}

The arguments in the next lemma were motivated by \cite{MR0457455}*{V~\S2}.

\begin{lemma} \label{L:Ribet}
Let $\lambda_1$ and $\lambda_2$ be distinct maximal ideals of $A$, and set $\aA=\lambda_1\lambda_2$.   Let $H$ be a subgroup of $\GL_2(A/\aA)$ for which the following hold:
\begin{alphenum}
\item \label{I:Ribet i} $\det(H) = (A/\aA)^\times$;
\item \label{I:Ribet ii} the projections $p_1'\colon H' \to \SL_2(\FF_{\lambda_1})$ and $p_2'\colon H' \to \SL_2(\FF_{\lambda_2})$ are surjective, where $H':= H \cap \SL_2(A/\aA)$;
\item \label{I:Ribet iii} the ring generated by the set 
\[
\mathcal{S}:=\{  \tr(h)^2/\det(h) : h \in H  \} \cup \{ \det(h)/ \tr(h)^2 : h \in H \text{ with } \tr(h) \in (A/\aA)^\times \}
\] 
is $A/\aA$.
\end{alphenum}
Then $H=\GL_2(A/\aA)$. 
\end{lemma}
\begin{proof}
Let $N_1'$ be the kernel of $p_2'$ and let $N_2'$ be the kernel of $p_1'$; we may view $N_i'$ as a normal subgroup of $\SL_2(\FF_{\lambda_i})$.
By Lemma~\ref{L:Goursat},  the image of $H'$ in $\SL_2(\FF_{\lambda_1})/N_1' \times\SL_2(\FF_{\lambda_2})/N_2'$ is the graph of a group isomorphism
\begin{equation} \label{E:SL isom}
\SL_2(\FF_{\lambda_1})/N_1' \xrightarrow{\sim} \SL_2(\FF_{\lambda_2})/N_2'.
\end{equation}
If $N_1'= \SL_2(\FF_{\lambda_1})$ (equivalently, $N_2'=\SL_2(\FF_{\lambda_2})$), then one has $H'= \SL_2(\FF_{\lambda_1}) \times \SL_2(\FF_{\lambda_2}) = \SL_2(A/\aA)$.  Using condition (\ref{I:Ribet i}), we deduce that $H=\GL_2(A/\aA)$.

We now assume that $N_i'$ is a \emph{proper} normal subgroup of $\SL_2(\FF_{\lambda_i})$ for $i=1,2$.  Using Lemma~\ref{L:SL2 commutator}, we find that $N_i' \subseteq \{\pm I\}$.  From (\ref{E:SL isom}) and cardinality considerations, we deduce that $N(\lambda_1)=N(\lambda_2)$ (equivalently, $\FF_{\lambda_1}$ and $\FF_{\lambda_2}$ are isomorphic fields).

For $i\in\{1,2\}$, define the projection $p_i \colon H \to \GL_2(\FF_{\lambda_i})$.  Let $N_1$ be the kernel of $p_2$ and let $N_2$ be the kernel of $p_1$; we may view $N_i$ as a normal subgroup of $\GL_2(\FF_{\lambda_i})$.    By Lemma~\ref{L:Goursat},  the image of $H$ in $\GL_2(\FF_{\lambda_1})/N_1 \times\GL_2(\FF_{\lambda_2})/N_2$ is the graph of a group isomorphism
\begin{equation} \label{E:GL isom}
\GL_2(\FF_{\lambda_1})/N_1\xrightarrow{\sim} \GL_2(\FF_{\lambda_2})/N_2.
\end{equation}
Since $N_i/N_i'$ and $N_i'$ are abelian, we find that $N_i$ is a solvable normal subgroup of $\GL_2(\FF_{\lambda_i})$.  It is then readily checked that $N_i$ must be contained in the group of diagonal matrices of $\GL_2(\FF_{\lambda_i}).$  By taking further quotients, we find that the image of $H$ in $\PGL_2(\FF_{\lambda_1}) \times \PGL_2(\FF_{\lambda_2})$ is the graph of an isomorphism
\[
\alpha\colon \PGL_2(\FF_{\lambda_1}) \xrightarrow{\sim} \PGL_2(\FF_{\lambda_2}).
\]
By Theorem~3 of Hua's supplement in \cite{MR606555}, $\alpha$ lifts to an isomorphism
\[
\widetilde{\alpha} \colon \GL_2(\FF_{\lambda_1}) \xrightarrow{\sim} \GL_2(\FF_{\lambda_2}).
\]

Let $\sigma \colon \FF_{\lambda_1}\xrightarrow{\sim} \FF_{\lambda_2}$ be a field isomorphism and $\chi\colon  \GL_2(\FF_{\lambda_1}) \to \FF_{\lambda_2}^\times$ a character;  these define two group homomorphisms $\GL_2(\FF_{\lambda_1}) \xrightarrow{\sim} \GL_2(\FF_{\lambda_2})$:
\begin{align} \label{E:Hua form}
A &\mapsto \chi(A) A^\sigma, \quad\quad A \mapsto \chi(A) ((A^{T})^{-1})^\sigma;
\end{align}
where $B^\sigma$ represents the matrix obtained by applying $\sigma$ to each entry of a matrix $B\in \GL_2(\FF_{\lambda_1})$.  By Theorem~1 of Hua's supplement in \cite{MR606555} (and using that $\FF_{\lambda_1}\cong \FF_{\lambda_2}$), we find that there are $\sigma$ and $\chi$ such that $\widetilde{\alpha}$ is the composition of an inner automorphism with one of the homomorphisms of (\ref{E:Hua form}).   We leave it to the reader to check that in either case, we have
\[
\frac{\tr(\widetilde{\alpha}(A))^2}{\det(\widetilde{\alpha}(A))} = \sigma \bigg( \frac{\tr(A)^2}{\det(A)} \bigg).
\]

Note that the map $\GL_2(\FF_{\lambda_i}) \to \FF_{\lambda_i}, \, A \mapsto \tr(A)^2/\det(A)$ factors through the projection $\GL_2(\FF_{\lambda_i})\to \PGL_2(\FF_{\lambda_i}).$  We deduce that  $\sigma( \tr(h_1)^2/\det(h_1) ) = \tr(h_2)^2/\det(h_2)$ for every $(h_1,h_2)\in H$.  Let $W$ be the ring of $(x_1,x_2) \in \FF_{\lambda_1}\times \FF_{\lambda_2}= A/\aA$ for which $\sigma(x_1)=x_2$.  We have just verified that $\mathcal{S}\subseteq W$.   However, $W\neq A/\aA$, and this contradicts assumption (\ref{I:Ribet iii}).
\end{proof}

\bibliographystyle{plain}

\begin{bibdiv}
\begin{biblist}

\bib{Magma}{article}{
      author={Bosma, Wieb},
      author={Cannon, John},
      author={Playoust, Catherine},
       title={The {M}agma algebra system. {I}. {T}he user language},
        date={1997},
     journal={J. Symbolic Comput.},
      volume={24},
      number={3-4},
       pages={235\ndash 265},
        note={Computational algebra and number theory (London, 1993)},
}

\bib{MR1185592}{article}{
      author={Brown, M.~L.},
       title={Singular moduli and supersingular moduli of {D}rinfeld modules},
        date={1992},
     journal={Invent. Math.},
      volume={110},
      number={2},
       pages={419\ndash 439},
}

\bib{MR2441247}{article}{
      author={Cojocaru, Alina~Carmen},
      author={David, Chantal},
       title={Frobenius fields for {D}rinfeld modules of rank 2},
        date={2008},
     journal={Compos. Math.},
      volume={144},
      number={4},
       pages={827\ndash 848},
}

\bib{MR2099195}{article}{
      author={Cojocaru, Alina~Carmen},
      author={Murty, M.~Ram},
       title={Cyclicity of elliptic curves modulo {$p$} and elliptic curve
  analogues of {L}innik's problem},
        date={2004},
     journal={Math. Ann.},
      volume={330},
      number={3},
       pages={601\ndash 625},
}


\bib{MR1373559}{article}{
      author={David, Chantal},
       title={Average distribution of supersingular {D}rinfel$'$d modules},
        date={1996},
     journal={J. Number Theory},
      volume={56},
      number={2},
       pages={366\ndash 380},
}

\bib{MR902591}{incollection}{
      author={Deligne, Pierre},
      author={Husemoller, Dale},
       title={Survey of {D}rinfel$'$d modules},
        date={1987},
   booktitle={Current trends in arithmetical algebraic geometry ({A}rcata,
  {C}alif., 1985)},
      series={Contemp. Math.},
      volume={67},
   publisher={Amer. Math. Soc.},
     address={Providence, RI},
       pages={25\ndash 91},
}

\bib{MR606555}{book}{
      author={Dieudonn{\'e}, Jean},
       title={On the automorphisms of the classical groups},
      series={Memoirs of the American Mathematical Society},
   publisher={American Mathematical Society},
     address={Providence, R.I.},
        date={1980},
      volume={2},
        note={With a supplement by Loo Keng Hua [Luo Geng Hua], Reprint of the
  1951 original},
}

\bib{MR0384707}{article}{
      author={Drinfel$'$d, V.~G.},
       title={Elliptic modules},
        date={1974},
     journal={Mat. Sb. (N.S.)},
      volume={94(136)},
       pages={594\ndash 627, 656},
}

\bib{MR2366959}{article}{
      author={Gekeler, Ernst-Ulrich},
       title={Frobenius distributions of {D}rinfeld modules over finite
  fields},
        date={2008},
     journal={Trans. Amer. Math. Soc.},
      volume={360},
      number={4},
       pages={1695\ndash 1721},
}

\bib{MR1423131}{book}{
      author={Goss, David},
       title={Basic structures of function field arithmetic},
      series={Ergebnisse der Mathematik und ihrer Grenzgebiete (3)},
   publisher={Springer-Verlag},
     address={Berlin},
        date={1996},
      volume={35},
}


\bib{MR2778661}{article}{
      author={Greicius, Aaron},
      title={Elliptic curves with surjective adelic {G}alois representations},
      date={2010},
      journal={Experiment. Math.},
      volume={19},
      number={4},
      pages={495\ndash 507},
}


\bib{MR1055716}{article}{
      author={Gupta, Rajiv},
      author={Murty, M.~Ram},
       title={Cyclicity and generation of points mod {$p$} on elliptic curves},
        date={1990},
     journal={Invent. Math.},
      volume={101},
      number={1},
       pages={225\ndash 235},
}


\bib{MR0330106}{article}{
      author={Hayes, D.~R.},
       title={Explicit class field theory for rational function fields},
        date={1974},
     journal={Trans. Amer. Math. Soc.},
      volume={189},
       pages={77\ndash 91},
}

\bib{MR1781330}{article}{
      author={Hsia, Liang-Chung},
      author={Yu, Jing},
       title={On characteristic polynomials of geometric {F}robenius associated
  to {D}rinfeld modules},
        date={2000},
     journal={Compositio Math.},
      volume={122},
      number={3},
       pages={261\ndash 280},
}

\bib{Jain-thesis2008}{thesis}{
      author={Jain, Lalit~Kumar},
       title={Koblitz's conjecture for the Drinfeld module},
        school = {University of Waterloo},
        type={Master's Thesis},
        date={2008},
}

\bib{MR2559117}{article}{
      author={Kuo, Wentang},
      author={Liu, Yu-Ru},
       title={Cyclicity of finite {D}rinfeld modules},
        date={2009},
     journal={J. Lond. Math. Soc. (2)},
      volume={80},
      number={3},
       pages={567\ndash 584},
}


\bib{MR0568299}{book}{
      author={Lang, Serge},
      author={Trotter, Hale},
       title={Frobenius distributions in {${\rm GL}_{2}$}-extensions},
      series={Lecture Notes in Mathematics, Vol. 504},
   publisher={Springer-Verlag},
     address={Berlin},
        date={1976},
        note={Distribution of Frobenius automorphisms in
  ${{\rm{G}}L}_{2}$-extensions of the rational numbers},
}

\bib{MR2488548}{article}{
      author={Lehmkuhl, Thomas},
       title={Compactification of the {D}rinfeld modular surfaces},
        date={2009},
     journal={Mem. Amer. Math. Soc.},
      volume={197},
      number={921},
       pages={xii+94},
}


\bib{MR698163}{article}{
      author={Murty, M.~Ram},
       title={On {A}rtin's conjecture},
        date={1983},
     journal={J. Number Theory},
      volume={16},
      number={2},
       pages={147\ndash 168},
}



\bib{MR2499412}{article}{
      author={Pink, Richard},
      author={R{\"u}tsche, Egon},
       title={Adelic openness for {D}rinfeld modules in generic
  characteristic},
        date={2009},
     journal={J. Number Theory},
      volume={129},
      number={4},
       pages={882\ndash 907},
}

\bib{MR2499411}{article}{
      author={Pink, Richard},
      author={R{\"u}tsche, Egon},
       title={Image of the group ring of the {G}alois representation associated
  to {D}rinfeld modules},
        date={2009},
     journal={J. Number Theory},
      volume={129},
      number={4},
       pages={866\ndash 881},
}

\bib{MR1606391}{article}{
      author={Poonen, Bjorn},
       title={Drinfeld modules with no supersingular primes},
        date={1998},
     journal={Internat. Math. Res. Notices},
      number={3},
       pages={151\ndash 159},
}

\bib{MR0457455}{article}{
      author={Ribet, Kenneth~A.},
       title={Galois action on division points of {A}belian varieties with real
  multiplications},
        date={1976},
     journal={Amer. J. Math.},
      volume={98},
      number={3},
       pages={751\ndash 804},
}

\bib{MR2020270}{article}{
      author={Rosen, Michael},
       title={Formal {D}rinfeld modules},
        date={2003},
     journal={J. Number Theory},
      volume={103},
      number={2},
       pages={234\ndash 256},
}

\bib{MR0263823}{book}{
      author={Serre, Jean-Pierre},
       title={Abelian {$l$}-adic representations and elliptic curves},
      series={McGill University lecture notes written with the collaboration of
  Willem Kuyk and John Labute},
   publisher={W. A. Benjamin, Inc., New York-Amsterdam},
        date={1968},
}

\bib{MR0387283}{article}{
      author={Serre, Jean-Pierre},
       title={Propri\'et\'es galoisiennes des points d'ordre fini des courbes
  elliptiques},
        date={1972},
     journal={Invent. Math.},
      volume={15},
      number={4},
       pages={259\ndash 331},
}

\bib{Serre-resume1977-1978}{article}{
      author={Serre, Jean-Pierre},
       title={R\'esum\'e des cours de 1977-1978},
        date={1978},
     journal={Annuaire du Coll\`ege de France},
       pages={67\ndash 70},
}

\bib{MR2018826}{article}{
      author={Yu, Jiu-Kang},
       title={A {S}ato-{T}ate law for {D}rinfeld modules},
        date={2003},
        ISSN={0010-437X},
     journal={Compositio Math.},
      volume={138},
      number={2},
       pages={189\ndash 197},
         url={http://dx.doi.org/10.1023/A:1026122808165},
}

\bib{Zywina-Koblitz}{article}{
      author={Zywina, David},
       title={A refinement of {K}oblitz's conjecture},
       volume={7},
       number={3},
       date={2011},
       pages={739--769},
        journal={Int. J. Number Theory},
}

\bib{Zywina-SatoTate}{article}{
      author={Zywina, David},
       title={{T}he {S}ato-{T}ate law for {D}rinfeld modules},
        date={2011},
        note={arXiv:1110.4098},
}

\end{biblist}
\end{bibdiv}


\end{document}